\documentclass[a4paper]{amsart}
\usepackage{textcomp, amsmath, amssymb, amsthm}
\usepackage{graphicx}
\usepackage{enumerate}

\newtheorem*{theorem*}{Theorem}
\newtheorem{theorem}{Theorem}[section]
\newtheorem{lemma}[theorem]{Lemma}

\newtheorem{cor}[theorem]{Corollary}

\theoremstyle{definition} \newtheorem{remark}[theorem]{Remark}
\newtheorem*{defin*}{Definition}

\newtheorem{notation}[theorem]{Notation}

\def\rr{{\mathbb R}}
\def\R{{\mathbb R}}
\def\Q{{\mathbb Q}}
\def\iK{{\mathcal{K}}}

\def\su{\subset}
\def\sp{\supset}

\def\al{\alpha}

\def\ga{\gamma}

\def\de{\delta}

\def\ep{\varepsilon}

\def\ol{\overline}
\def\ti{\widetilde}
\def\diam{{\rm diam}\, }
\def\dim{{\rm dim}\, }

\def\leb{\mathcal{L}}
\def\hau{\mathcal{H}}

\def\lkb{\lesssim}
\def\gkb{\gtrsim}
\def\cd{\cdot}
\providecommand{\semmi}[1]{}
\begin{document}

\title[Hausdorff dimension of unions of affine subspaces]{Hausdorff dimension of unions of affine subspaces and of Furstenberg-type sets}
\author{K. H\'era, T. Keleti, and A. M\'ath\'e}

\address
{Institute of Mathematics, E\"otv\"os Lor\'and University, 
P\'az\-m\'any P\'e\-ter s\'et\'any 1/c, H-1117 Budapest, Hungary}

\email{herakornelia@gmail.com}

\email{tamas.keleti@gmail.com}

\address
{Mathematics Institute, University of Warwick, Coventry, CV4 7AL, UK}

\email{a.mathe@warwick.ac.uk}

\thanks{This research was supported  
by the Hungarian National Research, Development and Innovation Office -- NKFIH, 104178, 
and the first author was also supported by the \'UNKP-16-3 New National Excellence Program of the Ministry of Human Capacities.} 

\begin{abstract}
We prove that for any $1 \leq k<n$ and $s\le 1$, the union of any nonempty 
$s$-Hausdorff dimensional family of $k$-dimensional affine subspaces 
of $\R^n$ has
Hausdorff dimension $k+s$.
More generally, we show that for any $0 < \al \leq k$, if $B \su \rr^n$ and $E$ is a nonempty collection 
of $k$-dimensional affine subspaces of $\R^n$ such that every
$P \in E$ intersects $B$ in a set of Hausdorff dimension at least $\al$, 
then 
$\dim B \ge 2 \al - k + \min(\dim E, 1)$, where $\dim$ denotes 
the Hausdorff dimension. 
As a consequence, we generalize the well-known Furstenberg-type estimate 
that every $\al$-Furstenberg set has Hausdorff dimension at least $2 \al$; we strengthen a theorem of Falconer and Mattila \cite{FaMa};
and we show that for any $0 \leq k<n$, if a set $A \su \rr^n$ contains the $k$-skeleton of a 
rotated unit cube around every point of $\rr^n$, 
or if $A$ contains a $k$-dimensional affine subspace at a fixed positive distance
from every point of $\R^n$, 
then the Hausdorff dimension of $A$ is at least $k + 1$. 
\end{abstract}

\maketitle

\section{Introduction}
There are several problems gathering around the general principle that
an $s$-dimensional collection of $d$-dimensional sets in $\rr^n$ must have positive measure
if $s+d > n$ and Hausdorff dimension $s+d$ if $s+d \leq n$, unless the sets have large
intersections. For example, Wolff \cite{Wo97,Wo00} proved that if a planar 
set $B$ contains
a circle around every point of a Borel set $S\su \R^2$ of Hausdorff dimension 
$s$ then $B$ has positive
Lebesgue measure provided $s >1$, and the Hausdorff dimension of $B$ is
at least $s+1$ when $s\le 1$. Most of these problems are only partially solved. 
The most famous example is the Kakeya conjecture, which states that 
every Besicovitch set (a compact set that contains a unit line segment in 
every direction) in $\R^n$ has Hausdorff dimension $n$, see e.g \cite{Ma15}. 
Note that the directions of lines of 
$\R^n$ form a set of dimension $n-1$, so the line segments of a Besicovitch
set form a collection 
of Hausdorff dimension at least $n-1$, so the above principle
would indeed imply the Kakeya conjecture. 
On the other hand, the following trivial example shows that 
this principle cannot be applied for every $s$-dimensional 
collection of lines: 
for any collection of lines of a fixed plane of $\R^3$  
the union clearly has Hausdorff dimension at most $2$, which is less than
$s+1$ if $s>1$.
In this paper we show that the above principle holds for any $s$-dimensional
collection 
of lines or even $k$-dimensional affine subspaces provided that $s\le 1$.

\begin{theorem}\label{t:simplest}
For any integers $1\le k<n$ and
$s\in [0,1]$ the union of any nonempty $s$-Hausdorff-dimensional family of 
$k$-dimensional affine subspaces
of $\R^n$ has Hausdorff dimension $s+k$.
\end{theorem}

For the special case $k=n-1$ this was proved by Oberlin \cite{Ob} for compact (or analytic) families of hyperplanes.  
He also proved \cite{Ob2} that for any integers $1\le k<n$ and any $s \geq 0$, the union of any nonempty compact (or analytic) $s$-Hausdorff-dimensional family of 
$k$-dimensional affine subspaces of $\R^n$ has Hausdorff dimension at least $\min\{n, 2k-k(n-k)+s\}$. 
Moreover, he proved that $s >(k+1)(n-k)-k$ implies positive Lebesgue measure for such unions, and the bound for $s$ is sharp. 
His results are in harmony with the above heuristic principle in the case of hyperplanes. 

 Falconer and Mattila \cite{FaMa} proved a stronger statement
 both in the $s\le 1$ and $s>1$ cases for hyperplanes: 
instead of full $n-1$-dimensional
affine subspaces it is enough to take a positive measure subset of each of them. 

We can go even further for any $k<n$: it is enough to take a $k$-Hausdorff
dimensional subset of each $k$-dimensional subspace:

\begin{theorem}
\label{t:FalcMatt}
Let $1 \leq k <n$ be integers and $s\in[0,1]$. 
If $E$ is a nonempty $s$-Hausdorff dimensional family of $k$-dimensional affine
subspaces and $B$ is a subset of $\bigcup_{P \in E} P$ such that $B\cap P$
has Hausdorff dimension $k$ for every $P\in E$ then
\begin{equation}\label{firstone}
 \dim  B= \dim \left( \bigcup_{P \in E} P \right) = s+k,
 \end{equation}
where here and in the sequel $\dim$ denotes Hausdorff dimension.
\end{theorem}

Note that this theorem does not assume any kind of measurability of $E$ (or $B$), unlike the mentioned results of Oberlin and Falconer and Mattila.

As we explained above, the equality on the right hand side of \eqref{firstone} is not true 
without the restriction $s\le 1$. 
But it is possible that the equality on the left hand side always holds. 
If it holds for $k=1$ in $\rr^n$ for all $n \geq 2$, 
it would imply that Besicovitch sets in $\rr^n$ have Hausdorff dimension at least $n-1$ and upper Minkowski dimension $n$, see \cite{Ke}.

In Theorem~\ref{t:FalcMatt}, 
$\dim  B \le \dim \left( \bigcup_{P \in E} P \right)$ is obvious,
$\dim\left( \bigcup_{P \in E} P \right) \le s+k$ is easy (Lemma~\ref{l:easy}),
the essence of the result is the estimate $\dim B\ge s+k$.

\medskip

It is natural to ask what happens if we go further 
and take the union of $\al$-Hausdorff dimensional subsets of 
an $s$-dimensional family of $k$-dimensional affine subspaces of $\R^n$ 
for some $s\le 1$ and $\al\in(0,k]$. Our most general result 
(Theorem~\ref{thm1}) gives
that in this case the union has Hausdorff dimension at least $2\al-k+s$.

Note that this result implies a known Furstenberg-type estimate. 
Let $0 < \al \leq 1$ and suppose that $F \su \rr^2$ is a Furstenberg set: 
a compact set such that for
every $e \in S^1$ there is a line $L_e$ in direction $e$ for which 
$\dim L_e \cap F \geq \al$, see e.g.~\cite{Ma15}. 
Since the sets $L_e$ form an at least $1$-dimensional collection of lines of $\R^2$
our above mentioned result gives $\dim F \geq 2 \al - 1 + 1= 2 \al$.

Molter and Rela \cite{MR} proved that if $E\su S^1$
has Hausdorff dimension $s$,  
$F\su\R^2$ and for every $e\in S^1$ there is a line 
$L_e$ in direction $e$ for which 
$\dim L_e \cap F \geq \al$ then $\dim F\ge 2\al-1+s$ and 
$\dim F\ge \al+\frac{s}2$. So our result is also a generalization of the
first estimate of Molter and Rela.

Our original motivation comes from the following question: What is the
minimal Hausdorff dimension of a set in $\R^n$ that contains the $k$-skeleton
of a rotated unit cube centered at every point of $\R^n$? In \cite{CCHK}
it is proved that for every $0\le k <n$ there exist such sets of
Hausdorff dimension at most $k+1$. 
As a fairly quick application of the above results we show 
(Theorem~\ref{fact2}) that for every $0\le k <n$ such a set must have
Hausdorff dimension at least $k+1$, so $k+1$ is the minimal Hausdorff 
dimension. 
We remark that if we have $k$-skeletons 
of \emph{rotated and scaled} cubes centered at every point 
then the minimal Hausdorff dimension is $k$, see \cite{CCHK}, and 
if we allow only \emph{scaled axis-parallel} cubes then the minimal 
Hausdorff dimension is $n-1$, see \cite{KNS} for $k=1, n=2$ and \cite{Th}
for the general case.
We also show (Theorem~\ref{fact2}) that 
if $A$ contains a $k$-dimensional affine subspace at a fixed positive distance
from every point of $\R^n$, 
then the Hausdorff dimension of $A$ is at least $k + 1$.

The paper is organized as follows: In Section \ref{mainn} we state 
our most general result (Theorem~\ref{thm1}), and prove its corollaries. 
In Section \ref{bas} we prove Theorem~\ref{thm1} subject to a lemma 
(Lemma \ref{Ade}), which will be proved in Sections \ref{codd} and \ref{ll2}.
In Section \ref{codd} we prove a purely geometrical lemma, 
which will be used during the $\mathit{L}^2$ estimation procedure 
in Section \ref{ll2} to prove Lemma~\ref{Ade}. 

\begin{notation}
For any integers $1 \leq  k < n$, let $A(n, k)$ denote the space
of all $k$-dimensional affine subspaces of $\rr^n$. 
For any $s \geq  0$, $\de \in (0,\infty]$ and $A \su \rr^n$, the $s$-dimensional Hausdorff $\de$-premeasure of $A$ will be denoted by $\hau^s_{\de}(A)$,  
the $s$-dimensional Hausdorff measure by $\hau^s(A)$, and 
the Hausdorff dimension of $A$ by $\dim A$. 
The open ball of center $x$ and radius $r$ will be denoted by $B(x,r)$ or $B_{\rho}(x,r)$ if we want to indicate the metric $\rho$. 
For a set $U \su \rr^n$, $U_{\de}=\cup_{x \in U} B(x,\de)$ denotes the open $\de$-neighborhood of $U$. 
We will use the notation $a \lesssim_{\alpha} b$ if $a \leq Cb$ where $C$ is a constant depending on $\alpha$. 
If it is clear from the context what $C$ should depend on, we may write only $a \lesssim b$. 
\end{notation}

\section{The most general theorem and its corollaries} 

Our most general result is the following: 
\label{mainn}
\begin{theorem}
\label{thm1}
Let $1 \leq  k < n$ be integers, let $A(n, k)$ denote the space
of all $k$-dimensional affine subspaces of $\rr^n$ and consider any natural metric on
$A(n, k)$. Let $0 < \al \leq k$ be any real number.
Suppose that $B \su \rr^n, \emptyset \neq E \su A(n, k)$ and for every $k$-dimensional affine subspace
$P \in E$, $\dim (P \cap B) \geq \al$. Then 
\begin{equation}
\label{geq}
\dim B \geq 2\al-k + \min(\dim E, 1). 
\end{equation}
\end{theorem}

\begin{remark}
An example for such a metric on $A(n,k)$ is defined in \cite{Ma}, p.~53. 
Let $\rho$ denote the given metric on $A(n,k)$. We say that $\rho$ is a natural metric if 
$\rho$ and the metric $d$ defined in \cite{Ma} are strongly equivalent; that is, 
there exist positive constants $K_1$ and $K_2$ such that, for every $P,P' \in A(n,k)$,
$K_1 \cdot d(P,P')\leq \rho (P,P') \leq K_2 \cdot d(P,P').$
\end{remark}

\begin{remark}
\label{rem}
For $\al=k$ and $\dim E\le k+1$ the estimate \eqref{geq} is sharp in the sense that for any $s\in[0,k+1]$ there exist sets $E$ and $B$ with the above property and $\dim E=s$
such that we have equality in \eqref{geq}: 
it is easy to see using
Theorem~\ref{t:FalcMatt} that we obtain such an example by letting $E$ to be
any $s$-Hausdorff dimensional
collection of $k$-dimensional affine subspaces of a fixed
$k+1$-dimensional subspace of $\R^n$ and $B=\cup_{P\in E} P$.

Clearly, \eqref{geq} can be a good estimate only when $\al$ is close to $k$:
for $\al<k-1$ the right-hand side of \eqref{geq} is less than $\al$ 
but trivially, $\dim B\ge \al$.
Since finding the best estimate for the $n=2, k=1, \dim E=1, \al<1$ case is 
essentially equivalent to finding the minimal Hausdorff dimension of 
a Furstenberg set, this cannot be easy and it is unlikely that
our estimate is sharp for any $\al<k$.
\end{remark}

By combining the $\al=k$ case of Theorem~\ref{thm1} and the following lemma,
we obtain Theorem~\ref{t:FalcMatt}, and its special case 
Theorem~\ref{t:simplest}.

\begin{lemma}\label{l:easy}
\label{leq}
For any $1 \leq  k < n$ integers and 
$\emptyset\neq E\su A(n,k)$ we have
$$
\dim \left(\bigcup_{P \in E} P \right)\leq k + \dim E.
$$
\end{lemma}

\begin{proof}
By taking a finite decomposition of $E$ if necessary,
we can assume that there exists a $P_0\in A(n,k)$
such that the orthogonal projection of $P_0$ onto any
$P\in E$ is $P$. Fix such a $P_0$. 
For any $P\in E$ and $t\in P_0$ let $h(P,t)$ be the 
orthogonal projection of $t$ onto $P$. 
Then $h(\{P\}\times P_0)=P$ for any $P\in E$, so
$h(E\times P_0)=\bigcup_{P \in E} P$. 
It is not hard to check that $h:E\times P_0\to\R^n$
is locally Lipschitz, therefore we obtain
$$
\dim\left(\bigcup_{P \in E} P\right) =
\dim(h(E\times P_0)) \le
\dim(E\times P_0) = \dim E + k.
$$
\end{proof}

Now we show a simple direct application of Theorem \ref{thm1}. 

\begin{cor}
\label{fact1}
Let $0 \leq k < n$ be integers, $0 \leq \al \leq k$, $\emptyset \neq C \su \rr^n$, and $B \su \rr^n$ 
such that for every $x \in C$ there exists a $k$-dimensional affine subspace $P$ containing $x$ 
such that $P$ intersects $B$ in a nonempty set of Hausdorff dimension at least $\al$. 
Then $\dim B \geq 2\al-k + \min(\dim C - k, 1)$.

Specially, if $1 \leq k$ and a set $A \su \rr^n$ contains a $k$-dimensional punctured affine subspace through every point of a set $C$ with $\dim C \geq k+1$, then
$\dim A \geq k+1$. 
\end{cor}

\begin{proof}
If $k=0$, or $k \geq 1$ and $\al=0$, then the statement clearly holds. Suppose now $k \geq 1, \al>0$. 
Let $E \su A(n,k)$ be the set of those $k$-dimensional affine subspaces that intersect
$B$ in a nonempty set of Hausdorff dimension at least $\al$. 
%
Then $C \su \bigcup_{P \in E} P$, thus 
$\dim C \leq \dim E + k$ by Lemma \ref{leq}, which means, $\dim E \geq \dim C-k$. 
Applying Theorem \ref{thm1} for $B$ and $E$, we obtain 
$\dim B \geq 2\al-k + \min(\dim C - k, 1)$.
\end{proof}

Our next goal is to show that if a set $B\su\R^n$ contains the
$k$-skeleton of a rotated unit cube around every point of $\R^n$ then
$\dim B\ge k+1$, as it was already stated in the Introduction.
Instead of the $k$-skeleton of the unit cube we will prove (Corollary \ref{ffact2}) the analogous result
for any $k$-Hausdorff dimensional set $S\su\R^n$ that can be covered by 
countably many $k$-dimensional affine subspaces. 
This result will follow from the following theorem.

\begin{theorem}
\label{fact2}
Let $0 \leq k < n$ be integers, $0 \leq \al \leq k$ and $0 \leq r$ be real numbers, $\emptyset \neq C \su \rr^n$, and $B \su \rr^n$ be
such that for every $x \in C$ there exists a $k$-dimensional affine subspace $P$
at distance $r$ from $x$ such that $P$ intersects $B$ in a nonempty set of Hausdorff dimension at least $\al$. 
Then $\dim B \geq 2\al-k + \dim C - (n-1)$.

Specially, if $B$ contains a $k$-dimensional affine subspace at a fixed positive distance from every point of $\R^n$, or 
if $B$ contains the $k$-skeleton of a rotated unit cube around every point of $\rr^n$, then $\dim B\ge k+1$.
\end{theorem}

\begin{proof}

If $r=0$, then we can apply Corollary \ref{fact1} and thus we get $\dim B \geq 2\al-k +\min(\dim C - k, 1) \geq 2\al-k +\dim C - (n-1)$. 

Suppose now that $r>0$.  
If $k=0$, then the condition of Theorem \ref{fact2} means that for every $x \in C$ there exists a point contained in $B$ at distance $r$ from $x$. 
Then $\bigcup_{p \in B} (p + rS^{n-1}) \sp C$, where $S^{n-1}$ denotes the unit sphere of center $0$ in $\rr^n$. 
Let $g: \rr^n \times S^{n-1} \to \rr^n$, $(p,e) \mapsto p+r e$. Clearly, $g$ is Lipschitz and  
$g(B \times S^{n-1})=\bigcup_{p \in B} (p + rS^{n-1})$. Thus we have 
$$\dim C \leq \dim \bigcup_{p \in B} (p + rS^{n-1}) \leq \dim (B \times rS^{n-1}) =\dim B + n-1,$$ thus 
$\dim B \geq \dim C - (n-1)$. 

If $k \geq 1$ and $\al=0$, then the statement is trivially true, 
so suppose now that $k \geq 1$, $\al >0$. We will use a similar argument as in the case $k=0$, but we use Theorem \ref{thm1}. 
Let $E \su A(n,k)$ be the set of those $k$-dimensional affine subspaces that intersect
$B$ in a set of Hausdorff dimension at least $\al$. 
By Theorem \ref{thm1} it
is enough to prove that $\dim E \geq \dim C - (n-1)$. 
For each $P \in E$ let $D(P) \su \rr^n$ be the union of those
$k$-dimensional affine subspaces that are parallel to $P$ and are at distance $r$ from $P$ (in
the Euclidean distance of $\rr^n$).
Clearly, $D(P)$ is exactly the set of those points of $\rr^n$ that
are at distance $r$ from $P$, thus by assumption, $\bigcup_{P \in E} D(P) \sp C$. 
It is easy to see that $\dim D(P)=n-1$ for any $P \in E$.

For any $P \in A(n,k)$, let $V_P$ denote the translate of $P$ containing $0$ and let $V_P^{\perp}$ denote the orthogonal complement of $V_P$. 
It is easy to see that there is a finite decomposition $E=\bigcup_{i=1}^{N} E_i$ such that for all $i$ 
there exists a $P_i \in A(n,k)$ with the following properties: 
the orthogonal projection of $P_i$ onto any
$P\in E_i$ is $P$, and the orthogonal projection of the $(n-k-1)$-sphere $V_{P_i}^{\perp} \cap S^{n-1}$ onto $V_P^{\perp}$ is contained in the
$\frac{1}{2}$-neighborhood of the $(n-k-1)$-sphere $V_P^{\perp} \cap S^{n-1}$, for any $P \in E_i$. 

Using the above properties, one can easily define for all $i$ a locally Lipschitz map  
$h_i: E_i \times D(P_i) \to \rr^n$ such that $h_i(\{P\} \times D(P_i))=D(P)$ for all $P \in E_i$. 
 We obtain 
$$\dim C \leq \dim \bigcup_{P \in E} D(P) =\max_i \dim \bigcup_{P \in E_i} D(P) = \max_i \dim h_i(E_i \times D(P_i)) \leq$$
$$ \leq \max_i \dim (E_i \times D(P_i)) = \max_i \dim E_i + n-1= \dim E + n-1,$$  
and thus $\dim E \geq \dim C - (n-1)$ and we are done. 
\end{proof}

\begin{remark}
In the special cases mentioned in Theorem~\ref{fact2}, the estimate is sharp. It is easy to see that
$B=\R^{k+1}\times \Q^{n-k-1}$ contains a $k$-dimensional affine subspace 
at every positive distance from every point of $\R^n$ and clearly $\dim B=k+1$. 
The construction given in \cite{CCHK} for a set $B$ with $\dim B =k+1$ containing the $k$-skeleton
of a rotated unit cube centered at every point of $\R^n$ is also based on this example.  
\end{remark}

\begin{cor}
\label{ffact2}
Let $0 \leq k < n$ be integers, $S \su \R^n$ with $\dim S=k$ that can be covered by a countable union
of $k$-dimensional affine subspaces. Let $\emptyset \neq C \su \R^n$, $A \su \R^n$ such that 
for all $x \in C$ there exists a rotation $T \in SO(n)$ such that $A$ contains $x + T(S)$. 
Then $\dim A \geq \max(k, k+ \dim C - (n-1))$. 
%
\end{cor}

\begin{proof}
Clearly, $\dim A \geq k$. Let $S_i \su \rr^n$, $i \geq 1$, be $k$-dimensional affine subspaces such that $S \su \bigcup_{i \geq 1} S_i$. 
Let $r_i=d(0,S_i)$, and
$\al_i=\dim (S_i \cap S)$. Then $\sup_{i \geq 1} \al_i =k$ by $\dim S=k$. The set $A$ has the property that for all $x \in \rr^n$, 
there exists an affine subspace $P=x+T(S_i)$ at distance $r_i$ from $x$ such that $\dim (A \cap P) \geq \al_i$, thus we can apply 
Theorem \ref{fact2} for each $i$. 
We obtain that $\dim A \geq 2\al_i-k + \dim C - (n-1)$ for all $i\geq 1$, and thus $\dim A \geq k + \dim C - (n-1)$. 
\end{proof}

\begin{remark}
The authors in \cite{CCHK} show that the estimate in Corollary \ref{ffact2} is sharp 
if $\dim C=n$ and $S$ can be covered by a countable union
of $k$-dimensional affine subspaces that do not contain the origin. 

On the other hand, if the covering subspaces contain the origin, then the estimate is not always sharp. 
Indeed, if $S$ is a punctured line through the origin and $C=\rr^n$, 
then $A$ is a Nikodym set, thus the conjecture is $\dim A = n$. 
The lower bounds obtained for the dimension of Besicovitch sets give lower bounds for the dimension of Nikodym sets, thus for $\dim A$ as well. 
A survey of the currently best lower bounds can be found in \cite{Ma15}. 
As an example, by \cite{Wo}, $\dim A \geq \frac{n+2}{2}$ which is better than the bound $2$ given by Corollary \ref{ffact2} provided $n >2$. 
\end{remark}

\section{The proof of Theorem \ref{thm1}}
\label{bas}

In this section we prove Theorem \ref{thm1} subject to a lemma (Lemma \ref{Ade}), which will be proved in Sections \ref{codd} and \ref{ll2}. 

We start with addressing measurability issues. 
For the definition of analytic sets, see e.g. \cite{Fr}.

\def\phi{\varphi}
\begin{lemma}
\label{use}
For $X\subset \R^n$, $\al>0$ and $c \geq 0$ let 
\begin{align*}
E_{{\alpha},c,X} & =\{P \in A(n,k) \colon\hau^\alpha_{\infty}(P\cap X) > c \}.
\end{align*}
If $X\subset \R^n$ is bounded $G_\delta$,
then  $E_{\alpha,c,X}$ is analytic. 
\end{lemma}

Lemma \ref{use} is an unpublished result of M. Elekes and Z. Vidny\'anszky. Similar statements were also proved in \cite{De}. 
For completeness, we include a proof here.

\begin{remark}
It is easy to see that if $X\subset \R^n$ is compact, then $E_{\alpha,c,X}$ is $F_{\sigma}$, thus also analytic. 
Therefore, to prove Theorem \ref{thm1} (or any of the above mentioned results) with the extra assumption that $B \su \rr^n$ is compact, the following argument could be skipped. 
\end{remark}

\begin{proof}
Let 
$$T=\{(P, x) \in A(n,k) \times \R^n \,:\, x\in P\},$$
this is the natural vector bundle of rank $k$ over $A(n,k)$. 
Let $\phi:T\to \R^n$ be defined by $\phi((P,x))=x$, and let $\pi:T\to A(n,k)$ be defined by $\pi((P,x))=P$. 
On $T$ we can consider the metric inherited from a product metric on $A(n,k)\times \R^n$ so that $\phi$ is isometry on all fibres.

Let $\iK$ be the space of those non-empty compact subsets of $T$ which lie in one fibre, that is,
$$\iK=\{K\subset T \,:\,K \text{ is non-empty compact, and }\pi(K)\text{ is a singleton}\}.$$
This is a complete metric space in the Hausdorff metric. 
Not to mix up singletons and their unique elements, let $\pi':\iK\to A(n,k)$ be defined by $\{\pi'(K)\}=\pi(K)$.

Since $X$ is $G_\delta$, $\phi^{-1}(X)$ is $G_\delta$ in $T$. It is easy to check that
$$\iK(\phi^{-1}(X)) \stackrel{\text{def}}{=} \{ K\in \iK \,:\,K\subset \phi^{-1}(X)\}$$
is also $G_\delta$ in $\iK$.

For $\al>0$ and $d > 0$, let
$$\iK^\alpha_d = \{ K \in \iK \,:\,\hau^\alpha_\infty(\phi(K))\ge d\}.$$
It is easy to see that these are closed sets in $\iK$.

Let
$$\iK_{\alpha, c, X} = \iK(\phi^{-1}(X)) \cap \bigcup_n \iK^\alpha_{c+1/n}.$$
Clearly, this is a Borel set in $\iK$. We claim that 

\begin{equation}
E_{\alpha, c, X}= \pi'(\iK_{\alpha, c, X}).
\label{eqmm1}
\end{equation}
Clearly, the right hand side consists of those $P\in A(n,k)$ for which $P\cap X$ contains a compact subset $K$ with $\hau^\alpha_\infty(K) > c$. 
We will show that for any $P \in A(n,k)$, 
\begin{equation}
\label{eqeq1}
\exists K \su P\cap X \ \text{compact with} \ \hau^\alpha_\infty(K) > c \ \Longleftrightarrow \ \hau^\alpha_\infty(P\cap X)>c,
\end{equation}
which implies \eqref{eqmm1}.

To prove \eqref{eqeq1}, we use the concept of capacities (see e.g.~\cite{Kec}, Section 30). 
\begin{defin*}
Let $Y$ be a Hausdorff topological space. A capacity on $Y$ is a map $\ga: \mathcal{P}(Y) \to [0,\infty]$ such that 
\begin{enumerate}[(i)]
	\item \label{one}
	$A \su B \ \Longrightarrow \ \ga(A) \leq \ga(B)$,
	
	\item \label{two}
	$A_0 \su A_1 \su \cdots \ \Longrightarrow \  \ga(A_n) \to  \ga(\cup_n A_n)$,
	
	\item \label{thr}
	for any compact $K \su Y$ we have  $\ga(K) < \infty$, and if $\ga(K) < r$, then for some open $U \sp K$, $\ga(U) < r$. 
\end{enumerate}
\end{defin*}
We claim that $\ga=\hau^\alpha_\infty$ is a capacity on $\ol{B(0,R)}$ for any $R>0$. Indeed, it is clear that $\hau^\alpha_\infty$ satisfies 
properties \eqref{one} and \eqref{thr} in any metric space, 
and it follows from the results in \cite{Da} that \eqref{two} holds for $\hau^\alpha_\infty$ in any compact metric space. 

Since $X$ is bounded $G_\delta$ (thus also analytic), and $\hau^\alpha_\infty$ is a capacity on the compact metric space $\ol{B(0,R)}$ with 
$X \su \ol{B(0,R)}$, the Choquet Capacitability Theorem (\cite{Kec}, (30.13)) can be applied, and it gives precisely \eqref{eqeq1}. 

Finally, \eqref{eqmm1} implies that $E_{\alpha, c, X}$ is a continuous image of a Borel set, thus analytic, and we are done.
\end{proof}

Note that the statement of Theorem \ref{thm1} is trivially true if $\dim E=0$, since $2 \al-k \leq \al$.

\begin{lemma}\label{assump}
Let $s=\min (\dim E,1) >0$. 
\label{ass}
We can make the following assumptions in the proof of Theorem \ref{thm1}:
\begin{enumerate}[(i)]

\item\label{gdelta}
$B$ is a $G_\delta$ set, that is, a countable intersection of open sets;

\item\label{poz}
$\hau^\alpha(P\cap B)>0$ for every $P\in E$;

\item
\label{bounded}
$B$ is bounded;

\item
\label{comp}\label{p2}
$E \su A(n,k)$ is compact, and $\hau^s(E)>0$. Moreover,
there is $\ep>0$ such that for every $P\in E$,
$$\hau^{\al}_{\infty}(P \cap B) \ge \ep.$$

\end{enumerate}
\end{lemma} 
Statement (\ref{poz}) is clearly weaker than (\ref{p2}); it is stated
to guide the proof.

\begin{proof} 

First we remark that if $E$ is replaced by any subset $\ti{E} \su E$, or $B$ is replaced by any superset $\ti{B} \supset B$, 
then the condition $\dim (P \cap \ti{B} ) \geq \dim (P \cap B ) \geq \al$ in Theorem \ref{thm1} is trivially satisfied for all $P \in \ti{E} \su E$. 

\begin{enumerate}[(i)]

\item
Let $\ti{B} \sp B$ be a $G_{\de}$ set with $\dim B= \dim \ti{B}$; the existence of such set is proved for example in \cite{Fr}. 
Clearly, it is enough to prove Theorem~\ref{thm1} for $\ti{B}$ replacing $B$.

\item
Clearly, replacing $\alpha$ with a slightly smaller value, we may assume, without loss of generality, that $\hau^{\alpha}(P\cap B)>0$ for every $P\in E$.

\item
If $B$ is not bounded then consider $B = \cup_n B_n$ where $B_n$ is bounded $G_{\de}$ and define $E_n=\{P \in E \colon \hau^{\alpha}(P\cap B_n)>0 \}$. 
Clearly, $E=\cup_n E_n$ thus $ \dim E=\sup\{ \dim E_n \colon n \in \mathbb{N}\}$. 
If Theorem~\ref{thm1} holds for the bounded set $B_n$ and $E_n \su A(n,k)$ for every $n$ then it holds for $B$ and $E$ as well. 
Thus we can assume that $B$ is bounded. 

\item
By \eqref{gdelta}, we may assume that $B$ is $G_\delta$.
By \eqref{poz}, for every $P\in E$, $\hau^\alpha(P\cap B)>0$, and thus $\hau^\alpha_\infty(P\cap B)>0$.
Thus $E\su \cup_{i=1}^\infty E_{\al,1/i,B}$, where the sets
$E_{\al,1/i,B}$ are the analytic sets given by Lemma~\ref{use}.
For every $\delta>0$, $\hau^{s-\delta}(E)=\infty$ and therefore
there is $i=i(\delta)$ with $\hau^{s-\delta}(E_{\al,1/i,B})>0$.
By Howroyd's theorem \cite{Ho}, there is a compact set
$E^\delta\subset E_{\al,1/i,B}$ with $\hau^{s-\delta} (E^\delta)>0$.
If Theorem~\ref{thm1} holds for these compact sets $E^\delta$, then
$\dim B \geq 2 \al - k+s-\delta$ for every $\delta>0$, which finishes the proof.
\end{enumerate}
\end{proof}

Let $e_0=(0,\dots,0)$; let $e_1=(1,0,\dots,0), \dots, e_n=(0,\dots,0,1)$ be the standard basis vectors of $\rr^n$, and 
let $V$ be the $k$-dimensional linear space generated by $e_1, \dots, e_k$. Put $H_0=V^{\bot}$, and  
$H_i=e_i + H_0$. Then $H_i$ is an $n-k$-dimensional affine subspace for all $i=1,\dots,k$. 
We use the sets $H_i$ ($i=0,\dots,k$) to describe the structure of $E$ by investigating the intersection of the elements of $E$ with them.

Let $C$ denote the convex hull of the vectors $e_0, e_1,\dots,e_k$ in $V$, $Q \su H_0$ the 
$n-k$-dimensional closed unit cube of center $e_0$ in $H_0$, and $S= C \times Q \su \rr^n$. 
Fix $\de_0>0$ and an open set $S'$ such that 
\begin{equation}
\label{sub}
S'_{\de_0} \subset S,
\end{equation}
 where $S'_{\de_0}$ denotes the open $\de_0$-neighborhood of $S'$.

\begin{lemma}\label{ci}
We can make the following further assumptions in Theorem \ref{thm1}.
\begin{enumerate}[(I)]
\item\label{p1} For every $P\in E$, $P \cap H_i$ is a singleton and contained in $S$ for all $i=0,1,\dots,k$;
\item \label{p3}   $B\subset S'$.
\end{enumerate}
\end{lemma}
\begin{proof} \ 
\begin{enumerate}[(I)]
\item We can cover $E$ by finitely many compact subsets for which \eqref{p1} holds after applying a suitable similarity transformation.
\item Since we may assume that $B$ is bounded, this can be obtained after applying a homothety.
\end{enumerate}
\end{proof}

Let us now fix $B$, $E$, $\ep$, $S'$, $\delta_0$ (and $s$ and $\alpha$) with properties given by Lemma~\ref{ass} and such that Lemma~\ref{ci} is satisfied. That is, $B$ is bounded and $G_\delta$, $E$ is compact and $\hau^s(E)>0$, and
\begin{equation}\label{a1}
\hau^{\al}_{\infty}(P \cap B) \geq \ep
\end{equation}
for all $P \in E$ for a fixed $\ep>0$.

We apply Frostman's lemma (see e.g. \cite{Ma}) 
to obtain a probability measure $\mu$ on $A(n,k)$ (for which Borel and analytic sets are measurable) supported on $E$
for which
\begin{equation}
\label{Frost}
\mu (B(P,r)) \lesssim r^s 
\end{equation}
for all $r>0$ and all $P \in E$. 

\def\eps{\ep}


Now we turn to estimating the dimension of the set $B$. Our aim is to show that
$$\hau^{2 \al-k+s-\ga}(B)>0$$
for any $\ga >0$. Fix $\ga>0$, and let $$u=2 \al-k+s-\ga.$$
Let $M$ be a positive integer such that
$$\sum_{k=M}^\infty 1/k^2 < \ep \quad \text{ and }\quad 2^{-M+1} \leq \de_0.$$ 
Let $B \su \bigcup_{i=1}^{\infty} B(x_i,r_i)$ be any countable cover with $2 r_i \leq 2^{-M}$ for all $i$.
For any $l \geq M$, let 
$$J_l=\{ i : 2^{-l} < r_i \leq 2^{-l+1} \}.$$ 
Let $R_l=\cup_{i \in J_l} B(x_i,r_i)$, and $B_l=R_l \cap B$. Then $B=\cup_{l=M}^{\infty} B_l$. 

Our aim is to find a big enough subset of $B$ that is covered by balls of approximately the same radii and such that many of the affine subspaces of $E$ have big intersection with it. 

\begin{remark}
\label{meas}
In the subsequent proofs, applications of Lemma~\ref{use} imply that the sets we take $\mu$-measure of are $\mu$-measurable, since they are in the $\sigma$-algebra generated by analytic sets. 
\end{remark}

\def\rc{l}

\begin{lemma}
\label{prob2}
There exists an integer $\rc \geq M$ such that 
\begin{equation}
\label{B_l}
\mu\left(P \in E \colon \hau^{\al}_{\infty}(P \cap B_{\rc}) \geq 
\frac{1}{\rc^2 }\right) \geq \frac{1}{\rc^2 }.
\end{equation}
\end{lemma}

\begin{proof}
Let $$A_l=\left\{P \in E \colon \hau^{\al}_{\infty}(P \cap B_{l}) \geq \frac{1}{\rc^2}\right\},$$ and 
assume that $\mu(A_{\rc}) < 1/{\rc}^2$ for all ${\rc} \geq M$.  
Since $\sum_{l=M}^\infty 1/l^2 < 1$ (we may assume $\ep\le 1$), these sets $A_{\rc}$ cannot cover $E$. Therefore, there exists $P\in E$
such that $\hau^{\al}_{\infty}(P \cap B_{\rc})< 1/l^2$, and thus
$\hau^{\al}_{\infty}(P \cap B)< \sum_{l=M}^\infty 1/l^2<\eps$, which contradicts \eqref{a1}.
\end{proof}

Fix the integer $\rc$ obtained by Lemma \ref{prob2} and let
\begin{equation}
\label{tie}
\ti{E}=A_{\rc}=\left\{P \in E \colon \hau^{\al}_{\infty}(P \cap B_{\rc}) \geq \frac{1}{\rc^2} \right\}.
\end{equation}
We will use the notation $\ti{P}=P \cap B_{\rc}$ for any $P \in \ti{E}$. 
We have 
\begin{equation}
\label{est}
\mu( \ti{E}) \geq \frac{1}{\rc^2} \ \textrm{and} \ \hau^{\al}_{\infty}(\ti{P})\geq \frac{1}{\rc^2}
\end{equation}
for every $P \in \ti{E}$ by Lemma \ref{prob2}. 
Note also that 
\begin{equation}
\label{bo2}
\ti{P}_{\de_0} \subset S
\end{equation}
for every $P \in \ti{E}$  by  \eqref{p3} of Lemma \ref{ci} and the definition of $S'$ and $\de_0$.

Let 
\begin{equation}
\label{Fdef}
F= \bigcup_{P \in \ti{E}} \ti{P} \su B_{\rc} \su B.
\end{equation}

Our aim is to find a lower estimate for $\leb^n (F_{\de})$. We will prove the following.

\begin{lemma}
There is a constant $c>0$ depending on $E$, $n$, and $k$ but independent of $\rc$, $\ep$, $\gamma$ and the covering of $B$ such that, 
for every $0<\delta\le \de_0$,
\label{Ade} 
$$\leb^n(F_{\de}) \ge c
\frac{\de^{n-(2\al-k+s)}}{\rc^8 \log \frac{1}{\de}}.
$$ 
\end{lemma}

\begin{remark}
Note that the integer $\rc$, the sets $\ti{E}$, $\ti{P}$ for every $P \in \ti{E}$, 
and $F$ depend on the cover $B \su \bigcup_{i=1}^{\infty} B(x_i,r_i)$.
\end{remark}
First we show how the proof can be finished using Lemma \ref{Ade}. We prove Lemma \ref{Ade} in Sections \ref{codd} and \ref{ll2}. 

\begin{remark}
As it happens often, it would be easier to prove the lower bound for the box dimension of $B$. 
For that purpose, we would not need the previous steps, it would be enough to estimate $\leb^n(B_{\de})$ from below. 
To prove the lower bound for the Hausdorff dimension, we sorted out a big enough part of $B$ that can be covered by balls of approximately the same radius. 
\end{remark}

Recall that $B \su \bigcup_{i=1}^{\infty} B(x_i,r_i)$, $2 r_i \leq 2^{-M} \leq \frac{\de_0}{2}$ for all $i$. 
We will use that the balls with indices from $J_{\rc}$ have approximately the same radius. We have that 
$$\sum_{i=1}^{\infty} (2 r_i)^u = \sum_{l=M}^{\infty} \sum_{i \in J_l} (2 r_i)^u \geq \sum_{i \in J_{\rc}} (2 r_i)^u \gkb 
\sum_{i \in J_{\rc }} (2^{-\rc})^u.$$
Let $\de=2^{-\rc+1}$. Then $\de \leq 2^{-M+1} \le \de_0$.
 
The set $F$ was constructed to satisfy 
$$F \su B_{\rc} \su \bigcup_{i \in J_{\rc}} B(x_i,r_i) \su \bigcup_{i \in J_{\rc}} B(x_i, \de),$$
and thus
$$F_{\de} \su \bigcup_{i \in J_{\rc}} B(x_i, 2 \de).$$

Using 
$$(2^{-\rc})^u \gkb \de^u = \frac{\de^n}{\de^{n-(2\al-k+s-\ga)}} \gkb \frac{\leb^n(B(x_i,2\de))}{\de^{n-(2\al-k+s-\ga)}}$$ 
and Lemma \ref{Ade}, we get 
\begin{align*}
\sum_{i \in J_{\rc}} (2^{-\rc})^u 
& \gkb \sum_{i \in J_{\rc}} 
 \frac{\leb^n(B(x_i,2\de))}{\de^{n-(2\al-k+s-\ga)}} 
 \geq 
\frac{\leb^n(F_{\de})}{\de^{n-(2\al-k+s)+\ga}} \gkb \\
& \gkb \frac{\de^{n-(2\al-k+s)}}{\de^{n-(2\al-k+s)+\ga} \rc^8 \log \frac{1}{\de}}=
\frac{1}{\de^{\ga} \rc^8 \log \frac{1}{\de}} \gkb \frac{1}{2^{-\rc \ga} \rc^9}.
\end{align*}
Thus we obtain
$$\inf_{\substack{B \su \bigcup_{i=1}^{\infty} B(x_i,r_i)\\ \forall i\ 2 r_i \leq 2^{-M}}} 
\sum_{i=1}^{\infty} (2 r_i)^u \gkb
\inf_{\rc \ge 1} \frac{1}{2^{-\rc \ga} \rc^9} \gkb_{\ga} 1$$
proving that $\hau^u(B)>0$ and we are done. 

\section{Geometric arguments}
\label{codd}
Now we start proving Lemma \ref{Ade}. In this section we prove a purely geometric lemma using only the set $\ti{E} \su A(n,k)$. 
This part is independent of the set $B$ and the number $\al$. 

\begin{lemma}
\label{gengeo}
For any $P, P' \in \ti{E}$, 
\begin{equation}
\label{zz1}
\leb^n(P_{\de} \cap P'_{\de} \cap S) \lkb \frac{\de^{n-k+1}}{\rho(P,P')+\de}
\end{equation}
 for all $0 < \de \leq \de_0$, where $\rho$ denotes the metric on $A(n,k)$, and $\de_0$ is from \eqref{sub}.

\end{lemma}

To prove Lemma \ref{gengeo}, we will define a new metric on $\ti{E}$ by making use of \eqref{p1} of Lemma \ref{ci}.
We will assign a code to each $k$-dimensional affine subspace in $\ti{E}$. 
For a given $P \in \ti{E}$, let $(0,a^0)=(0,\dots,0,a_1^0,\dots, a_{n-k}^0)$ denote  
the standard $\R^n$-coordinates of $P \cap H_0$. 
Similarly, let $(1^l,a^l)=(0,\dots,1,\dots,0, a_1^l,\dots, a_{n-k}^l)$ denote 
the standard $\R^n$-coordinates of $P \cap H_l$ 
for each $l=1,\dots,k$. 
Let $b^l=a^l-a^0 \in \R^{n-k}$ for each $l=1,\dots,k$. 
We refer to $a^0$ as the vertical intercept, and to $\{b^l\}_{l=1}^k$ as the slopes of $P$. 

We say that the point $x=x(P)=(a^0,b^1,\dots,b^k)=(a,b) \in \rr^{(k+1)(n-k)}$ is the \emph{code} 
of the $k$-dimensional affine subspace $P \in \ti{E}$. 
By \eqref{p1} of Lemma \ref{ci} one can see that $P \to x(P)$ is well defined and injective on $\ti{E}$.

We will use the maximum metric on the code space $\rr^{(k+1)(n-k)}$. 
This means, $\|x-x' \|=\max(\|a-a'\|, \|b-b'\|)$, where
$$\|a-a'\|=\max\limits_{j=1,\dots,n-k} |a_j^0-a_j^{0'}|,$$ 
\begin{equation}
\label{cb}
\|b-b'\|=\max\limits_{j=1,\dots,n-k} \left(\max\limits_{l=1,\dots,k} |b_j^l-b_j^{l'}|) \right). 
\end{equation}

\begin{remark}
\label{metric}
Put $d(P,P')=\|x(P)-x(P') \|$, then $d$ is a natural metric on $\ti{E}$. 
Thus the metrics $d$ and $\rho$ are strongly equivalent, this means, 
there exist positive constants $K_1$ and $K_2$ such that, for every $P,P' \in \ti{E}$,
$K_1 \cdot d(P,P')\leq \rho (P,P') \leq K_2 \cdot d(P,P').$
\end{remark}

%
In order to prove Lemma \ref{gengeo} we show that 
if $P$ and $P'$ are translated along $H_0$ far enough from each other  
compared to their slopes, then the intersection of their $\de$-tubes is empty in $S$, and 
if the slopes of $P$ and $P'$ are far enough from each other,  
then the intersection of their $\de$-tubes is small enough in $S$. 

\begin{lemma}
\label{code}
\begin{enumerate}[(a)]

\item
\label{empty}
There is a constant $D>0$ (depending only on $n$ and $k$) such that if $$\| a-a' \|  > \| b - b' \| + D \de $$
then $P_{\de} \cap P'_{\de} \cap S = \emptyset$ for all $0 < \de \leq \de_0$. 

\item
\label{small}
If $\|b-b'\| >0$, 
then $\leb^n(P_{\de} \cap P'_{\de} \cap S) \lkb \frac{\de^{n-k+1}}{\|b-b'\|}$ for all $0 < \de \leq \de_0$.

\end{enumerate}
\end{lemma}

\begin{proof}
Fix $P, P' \in \ti{E}$, and 
put $f,g: \rr^k \to \rr^{n-k}$, 
$$ t=(t_1,\dots,t_k) \mapsto a^0+t_1b^1+\dots+t_kb^k=f(t), $$
$$t=(t_1,\dots,t_k) \mapsto a^{0'}+t_1b^{1'}+\dots+t_kb^{k'}=g(t),$$ where 
$a^0,b^1,\dots b^k$, and $a^{0'},b^{1'},\dots b^{k'}$ are the code coordinates of $P$ and $P'$, respectively. 
Then
$$P \cap S = \{(t,f(t)) \in \rr^n \colon t \in C \}, \ P' \cap S = \{(t,g(t)) \in \rr^n \colon t \in C \}.$$ 

One can easily prove using \eqref{p1} of Lemma \ref{ci} and the compactness of $S$, that 
there is a constant $c>0$ independent of $\de$ such that for all $Q \in \ti{E}$, 
\begin{equation}
\label{tc}
Q_\de \su Q+ \left(\{0\} \times (-c\de,c\de)^{n-k} \right),
\end{equation}
where for $A,B \su \rr^n$, $A+B=\{a+b \colon a \in A, b \in B \}$. Fix such a constant $c$. 


Applying \eqref{tc} for $P$ and $P'$, we have
$$P_{\de} \su \{(t,u) \in \R^n \colon |f(t)-u| < c\de \}, 
P'_{\de} \su \{(t,u) \in \R^n \colon |g(t)-u| < c\de \},$$
and
\begin{equation}
\label{ball}
P_{\de} \cap P'_{\de} \cap S  \su \{(t,u) \in \R^n \colon u \in (B(f(t),c\de) \cap B(g(t),c\de)), t \in C \}.
\end{equation} 
Clearly, $|f(t)-g(t)|>2 c \de$ implies $B(f(t),c\de) \cap B(g(t),c\de)=\emptyset$, thus  $P_{\de} \cap P'_{\de} \cap S =\emptyset$. 
Put $D=2 c$, then $\| a-a' \|  > \| b - b' \| + D \de $ implies $|f(t)-g(t)|>2 c \de$, thus we are done with the proof of \eqref{empty} of Lemma \ref{code}.

By \eqref{ball} and Fubini's theorem we also have 
$$\leb^n(P_{\de} \cap P'_{\de} \cap S) \leq \int_C \leb^{n-k}(B(f(t),c\de) \cap B(g(t),c\de)) d\leb^k(t).$$
If $B(f(t),c\de) \cap B(g(t),c\de) \neq \emptyset$, we will use the trivial estimate 
$$\leb^{n-k}(B(f(t),c\de) \cap B(g(t),c\de)) \lkb_{n,k} (c\de)^{n-k} \lkb \de^{n-k}.$$ 
Put $N=\{t \in C \colon |f(t)-g(t)| \leq 2 c \de \}$, then 
$$\leb^n(P_{\de} \cap P'_{\de} \cap S) \lkb \int\limits_{N} \de^{n-k} d\leb^k(t).$$ 
%
Clearly, we have
\begin{equation}
\label{ph}
N \su \bigcap_{j=1}^{n-k} \left\{t \in C \colon 
\left|\left(a_j^0-a_j^{0'} \right) + \sum_{i=1}^k t_i \left(b_j^i-b_j^{i'} \right) \right| \leq 2c\de \right\}.
\end{equation}

By the definition of $\| \cdot \|$, there are indices $i,j$ such that $0 <\|b-b'\| =|b_j^i-b_j^{i'}|$. 
Fix such an $i$ and $j$, we can assume that $i=k$ without loss of generality. 
Then we get using \eqref{ph} that 
$$N \su \left\{t \in C \colon p_-(t) \leq t_k \leq p_+(t) \right\},$$
where 
$$p_-(t)=p_-(t_1,\dots,t_{k-1})=\frac{-2c\de-(a^0-a^{0'})-\sum\limits_{i=1}^{k-1} t_i (b_j^i-b_j^{i'})}{b_j^k-b_j^{k'}},$$ 
$$p_+(t)=p_+(t_1,\dots,t_{k-1})=\frac{2c\de-(a^0-a^{0'})-\sum\limits_{i=1}^{k-1} t_i (b_j^i-b_j^{i'})}{b_j^k-b_j^{k'}}.$$ 

The set $\left\{t \in C \colon p_-(t) \leq t_k \leq p_+(t) \right\}$ is obtained as the intersection of the simplex $C$ and 
the strip between the parallel hyperplanes $\{t_k=p_+(t)\}, \{t_k=p_-(t)\}$. 
One can easily calculate the distance of these hyperplanes, using the normal vector $n=(b_j^1-b_j^{1'},\dots,b_j^k-b_j^{k'})$. 
One gets  
$$d=d(\{t_k=p_+(t)\}, \{t_k=p_-(t)\})=\frac{2c\de}{\sqrt{\sum\limits_{i=1}^{k} (b_j^i-b_j^{i'})^2}}.$$
Thus the set $N$ is contained in a rectangular box, where the shortest side length is $d$ and the others are $\diam(C)=\sqrt{2}$. Then
$$\leb^k(N) \lkb \frac{\de}{\sqrt{\sum\limits_{i=1}^{k} (b_j^i-b_j^{i'})^2}} \lkb 
\frac{\de}{|b_j^k-b_j^{k'}|} = \frac{\de}{\|b-b'\|},$$
thus 
$$\leb^n(P_{\de} \cap P'_{\de} \cap S) \lkb \frac{\de^{n-k+1}}{\|b-b'\|}$$ 
 and we are done with the proof of Lemma \ref{code}. 
\end{proof}

Now we prove Lemma \ref{gengeo}. 
\begin{proof}
Using \eqref{empty} of Lemma \ref{code} we obtain that $P_{\de} \cap P'_{\de} \cap S= \emptyset$ for all $0 < \de \leq \de_0$
if $\| a-a' \|  > \| b - b' \| + D \de$, so \eqref{zz1} is clearly satisfied. 

Assume now that $\| a-a' \|  \leq \| b - b' \| + D \de$, and $\| b - b' \| \leq \de$. By Remark \ref{metric}, 
$$\rho (P,P') \leq 
K_2 \|x-x' \| \leq K_2 (D+1) \de \lkb \de,$$ 
and then since $S$ is bounded, we have
$$\leb^n(P_{\de} \cap P'_{\de} \cap S) \leq \leb^n(P_{\de} \cap S) \lkb \de^{n-k}=\frac{\de^{n-k+1} }{\de} \lkb \frac{\de^{n-k+1} }{\rho(P,P') + \de}.$$ 
Thus we are done in this case. 

If  $\| a-a' \|  \leq \| b - b' \| + D \de$ and $\|b-b'\| \geq \de$, we have that 
$$\rho (P,P') + \de \leq 
K_2 \|x-x' \| + \de \leq K_2 (\| b - b' \| + D \de)  + \de \lkb \|b-b'\|$$
using Remark \ref{metric} again.
Applying \eqref{small} of Lemma \ref{code}, we obtain that 
$$\leb^n(P_{\de} \cap P'_{\de} \cap S) \lkb \frac{\de^{n-k+1} }{\|b-b'\|} \lkb \frac{\de^{n-k+1} }{\rho(P,P') + \de},$$ 
which is \eqref{zz1}. 

\end{proof}

\section{The proof of Lemma \ref{Ade}, $\mathit{L}^2$ argument}
\label{ll2}

In this section we prove Lemma \ref{Ade} with help on an $\mathit{L}^2$ estimation technique. 
It resembles the technique that C\'ordoba used in his proof for the Kakeya maximal inequality in the plane, see \cite{Co}.  

By Fubini's theorem we have the following:
\begin{align*}
\int\limits_{\ti{E}} \leb^n (\ti{P}_{\de}) d\mu(P) & =\int\limits_{\ti{E}} \int\limits_{\rr^n} \chi_{\ti{P}_{\de}}(y) dy d\mu(P) \\
& = \int\limits_{\rr^n} \int\limits_{\ti{E}} \chi_{\ti{P}_{\de}}(y) d\mu(P) dy = \int\limits_{\rr^n}  \chi_{F_{\de}}(y) \cdot 
\int\limits_{\ti{E}} \chi_{\ti{P}_{\de}}(y) d\mu(P) dy,
\end{align*}
 where $F=\bigcup_{P \in \ti{E}} \ti{P}$ from \eqref{Fdef}.
Now we apply the Cauchy-Schwarz inequality for the $\mathit{L}^2$ functions $y \mapsto \chi_{F_{\de}}(y)$ and 
$y \mapsto \int\limits_{\ti{E}} \chi_{\ti{P}_{\de}}(y) d\mu(P)$. We get 
\begin{align*}
\int\limits_{\rr^n}  \chi_{F_{\de}} (y) \cdot \left(\int\limits_{\ti{E}} \chi_{\ti{P}_{\de}}(y) d\mu(P) \right) dy
& \leq 
\left(\int\limits_{\rr^n} \chi_{F_{\de}}^2 (y) dy \right)^{1/2} \cdot  
\left(\int\limits_{\rr^n} \left( \int\limits_{\ti{E}} \chi_{\ti{P}_{\de}}(y) d\mu(P) \right)^2 dy \right) ^{1/2} \\
& =
\left(\leb^n(F_{\de}) \right)^{1/2} \cdot 
\left(\int\limits_{\rr^n} \iint\limits_{\ti{E} \times \ti{E}} \chi_{\ti{P}_{\de}}(y) \cdot \chi_{\ti{P'}_{\de}}(y) d\mu(P) d\mu(P') dy \right)^{1/2} \\
& =\left(\leb^n(F_{\de})\right)^{1/2} \cdot 
\left( \iint\limits_{\ti{E} \times \ti{E}} \leb^n \left(\ti{P}_{\de} \cap \ti{P'}_{\de} \right) d\mu(P) d\mu(P')\right)^{1/2}.
\end{align*}

We proved that 
$$
\int\limits_{\ti{E}} \leb^n (\ti{P}_{\de}) d\mu(P)  \leq 
\left(\leb^n(F_{\de})\right)^{1/2} \cdot 
\left( \iint\limits_{\ti{E} \times \ti{E}} \leb^n \left(\ti{P}_{\de} \cap \ti{P'}_{\de} \right) d\mu(P)
 d\mu(P')\right)^{1/2}.
$$
On the other hand, there is a lower bound for the left hand side. 

For any $U \su \rr^n$ and $\ep>0$, let $N(U,\ep)$ denote the smallest number of $\ep$-balls needed to cover $U$. 
It is well known (see e.g. \cite{Ma}) that $\leb^n(U_{\ep}) \gkb_n N(U,2\ep) \ep^n$ for every $U \su \rr^n$ and $\ep>0$. 
Since $\hau^{\al}_{\infty}(\ti{P})\geq \frac{1}{\rc^2}$ by \eqref{est}, 
we have $N(\ti{P},\ep) \cdot (2\ep)^{\al} \geq \frac{1}{\rc^2}$ for every $P \in \ti{E}$ and $\ep>0$, thus 
\begin{equation}
\label{mink}
\leb^n(\ti{P}_{\de}) \gkb N(\ti{P},2\de) \cdot \de^n \gkb \de^{n-\al} \cdot \frac{1}{\rc^2}.
\end{equation}
Then 
$$\int\limits_{\ti{E}} \leb^n (\ti{P}_{\de}) d\mu(P) \gkb \de^{n-\al} \cdot \frac{1}{\rc^2} \cdot\mu(\ti{E}) \geq 
\de^{n-\al} \cd  \frac{1}{\rc^4}$$
by \eqref{est}. 

Thus we get
\begin{equation}
\label{l2}
\frac{\de^{2n-2\al}}{\rc^8} \lkb
\leb^n(F_{\de})
\cdot 
\iint\limits_{\ti{E} \times \ti{E}} \leb^n \left(\ti{P}_{\de} \cap \ti{P'}_{\de} \right) d\mu(P) d\mu(P').
\end{equation}

Thus we need to find an upper estimate for 
\begin{equation}
\label{double}
 \iint\limits_{\ti{E} \times \ti{E}} \leb^n (\ti{P}_{\de} \cap \ti{P'}_{\de}) d\mu(P) d\mu(P').
\end{equation}
This means, we have to investigate, how the different $\de$-tubes intersect each other. 

We will estimate the integral \eqref{double}
by dividing the set $\ti{E}$ into parts. 
One can easily check using Remark \ref{meas} that the elements of this partition will be measurable. 

Fix a $P' \in \ti{E}$. 
Put 
$$E_0=\{P \in \ti{E} \colon \rho(P,P') \leq \de \}$$ and 
$$E_j=\{P \in \ti{E} \colon 2^{j-1} \de < \rho(P,P') \leq 2^{j} \de \}$$ 
for $j=1,\dots, N$, where $N \lkb \log \frac{1}{\de}$. 
Clearly, we have $\ti{E}=\bigcup_{j=0}^{N} E_j$, so

\begin{equation}
\label{integral}
\int\limits_{\ti{E}} \leb^n (\ti{P}_{\de} \cap \ti{P'}_{\de}) d\mu(P)=
\int\limits_{E_0} \leb^n (\ti{P}_{\de} \cap \ti{P'}_{\de}) d\mu(P)+
\sum_{j=1}^{ N } \int\limits_{E_j} \leb^n (\ti{P}_{\de} \cap \ti{P'}_{\de}) d\mu(P).
\end{equation}

By \eqref{bo2} and Lemma \ref{gengeo}, we obtain
\begin{equation}
\label{zz2}
\int\limits_{E_j} \leb^n (\ti{P}_{\de} \cap \ti{P'}_{\de}) d\mu(P) \leq 
\int\limits_{E_j} \leb^n (P_{\de} \cap P'_{\de} \cap S) d\mu(P) \lkb \int\limits_{E_j} \frac{\de^{n-k+1}}{\rho(P,P')+\de} d\mu(P)
\end{equation}
for each $j=0,\dots,N$. 

In the case $j=0$, we obtain that 
\begin{equation}
\label{zz3}
\int\limits_{E_0} \leb^n (\ti{P}_{\de} \cap \ti{P'}_{\de}) d\mu(P) \lkb \frac{\de^{n-k+1}}{\de} \mu(E_0) \lkb \de^{n-k} \cdot \de^s
\end{equation}
by \eqref{zz2} and \eqref{Frost}. 

For $j \in \{1,\dots,N\}$, we get 
\begin{equation}
\label{zz4}
\int\limits_{E_j} \leb^n (\ti{P}_{\de} \cap \ti{P'}_{\de}) d\mu(P) \leq \frac{\de^{n-k+1}}{2^{j-1}\de +\de} \mu(E_j) \lkb \frac{\de^{n-k}}{2^j} (2^j\de)^s
= \frac{\de^{n-k+s} \cdot 2^{js}}{2^j} 
\end{equation}
by \eqref{zz2} and \eqref{Frost} again. 
Applying these estimates for \eqref{integral} and using $s \leq 1$, we get
$$\int\limits_{\ti{E}} \leb^n (\ti{P}_{\de} \cap \ti{P'}_{\de}) d\mu(P) \lkb \de^{n-k+s} \left(1+ \sum_{j=1}^{ N } \frac{2^{js}}{2^j} \right) 
\lkb \de^{n-k+s} N \lkb \de^{n-k+s} \log \frac{1}{\de}.$$

Finally we integrate with respect to $P'$ and obtain by $\mu(\ti{E}) \leq 1$ that 
$$\iint\limits_{\ti{E} \times \ti{E}} \leb^n (\ti{P}_{\de} \cap \ti{P'}_{\de}) d\mu(P) d\mu(P') \lkb \de^{n-k+s} \log \frac{1}{\de}.$$ 

Recalling \eqref{l2}, we obtain that 
$$\frac{\de^{2n-2\al}}{\rc^8} \lkb
\leb^n(F_{\de})
\cdot 
\de^{n-k+s} \log \frac{1}{\de},$$ 
 thus
$$\leb^n(F_{\de}) \gkb 
\frac{\de^{n-2\al+k-s}}{\rc^8 \log \frac{1}{\de}}$$ 
and we are done with the proof of Lemma \ref{Ade} and so also with the proof of Theorem \ref{thm1}. 

\section*{Acknowledgement}
The first author is grateful to Izabella {\L}aba for her notes on 
intersections of Cantor sets and differentiation for self-similar measures. 
It helped her choosing the formulation of the $\mathit{L}^2$ argument used in Section \ref{ll2}. 
The authors are grateful to the referees for their helpful suggestions.

\end{document}